\documentclass[12pt,a4paper]{article}
\usepackage{graphicx}
\usepackage[]{amsmath,amssymb,amsthm}
\usepackage{epstopdf, color}
\usepackage{setspace}

\newtheorem{theorem}{Theorem}[section]
\newtheorem{proposition}[theorem]{Proposition}
\newtheorem{lemma}[theorem]{Lemma}
\newtheorem{conj}[theorem]{Conjecture}

\theoremstyle{definition}
\newtheorem{definition}[theorem]{Definition}
\theoremstyle{remark}

\newcommand{\R}{\mathbb{R}}

\newcommand{\Z}{\mathbb{Z}}

\begin{document}

\title{Almost complete and equable heteroclinic networks}

\author{Peter Ashwin\thanks{Centre for Systems, Dynamics and Control, Department of Mathematics, University of Exeter, Exeter EX4 4QF, UK}, Sofia B.S.D. Castro\thanks{Faculdade de Economia and Centro de Matem\'atica, Universidade do Porto, Rua Dr.\ Roberto Frias, 4200-464 Porto, Portugal} \& Alexander Lohse\thanks{Fachbereich Mathematik, Universit\"at Hamburg, Bundesstra{\ss}e 55, 20146 Hamburg, Germany}~\thanks{Corresponding author, mailto:alexander.lohse@math.uni-hamburg.de}}

\maketitle

\begin{abstract}
Heteroclinic connections are trajectories that link invariant sets for an autonomous dynamical flow: these connections can robustly form networks between equilibria, for systems with flow-invariant spaces. In this paper we examine the relation between the heteroclinic network as a flow-invariant set and directed graphs of possible connections between nodes. We consider realizations of a large class of transitive digraphs as robust heteroclinic networks and show that although robust realizations are typically not complete (i.e.\ not all unstable manifolds of nodes are part of the network), they can be almost complete (i.e.\ complete up to a set of zero measure within the unstable manifold) and equable (i.e.\ all sets of connections from a node have the same dimension). We show there are almost complete and equable realizations that can be closed by adding a number of extra nodes and connections. 
We discuss some examples and describe a sense in which an equable almost complete network embedding is an optimal description of stochastically perturbed motion on the network.
\end{abstract}


\vspace{.3cm}



\section{Introduction}

Heteroclinic cycles and networks appear in a range of dynamical models posed as ordinary differential equations (ODEs) that try to capture ``intermittent'' behaviour, for example, in the onset of fluid turbulence, encoding of neural states or species competition in ecosystems: see for example Krupa \cite{Kru97}. They manifest as attracting dynamics where the state remains close to saddle equilibria for long periods of time, interspersed with rapid switches between equilibria. This behaviour can remain robust to perturbations that preserve some symmetries or other structures of the system: see for example Weinberger and Ashwin \cite{Weinberger2018} for a recent review.

In most cases, heteroclinic networks have been found and studied from analysis of a given system of ODEs. However, in an attempt to understand general properties of heteroclinic networks Ashwin and Postlethwaite \cite{AshPos13} suggest that the converse problem of {\em designing a system of ODEs that realize (i.e.\ embed) a given directed graph as a heteroclinic network} is of interest. It is also of potential interest in applications such as design of computational systems that permit only certain transitions. Several recent papers, Ashwin and Postlethwaite  \cite{AshPos13,AshPos16a}, Bick \cite{Bic18} and Field \cite{Fie15, Fie17}, have considered several approaches to the design of systems that have specific heteroclinic networks. These approaches to the realization of a graph as a heteroclinic network typically result in networks that are not asymptotically stable or even contain unstable manifolds of all saddles. This is discussed in \cite{Fie17} where a heteroclinic network is called {\em clean} if it is compact and equal to the union of the unstable manifolds of its equilibria. In the present article, we consider networks that are typically not compact -- we call them {\em complete} if they contain all unstable manifolds of their equilibria. Thus, a network is clean if and only if it is compact and complete. The notion of completeness is related to whether the network can be visible as an attractor: indeed it is necessary for a network to be clean/complete for it to be asymptotically stable \cite[Remark 1.4]{Fie17}.

This paper introduces some concepts, results and examples that aim to clarify the structure and dynamics of heteroclinic networks by showing that although we typically cannot realize arbitrary directed graphs (from a large class) as clean heteroclinic networks, we can achieve {\em almost completeness} (the network contains almost all of the unstable manifolds) and in addition ensure {\em equability} (a property of a node meaning that all outgoing connections from that node have the same dimension) of all nodes in the network.

To introduce this more precisely, we consider a system of ordinary differential equations 
\begin{equation}
\dot{x}=f(x)
\label{eq:ode}
\end{equation}
in $\R^n$ ($n<\infty$) with smooth $f$ and a bounded globally attracting open set: we will write $\varphi(t,x_0)$ to denote the flow generated by solutions $x(t)$ of (\ref{eq:ode}) starting at $x_0$. Clearly a variety of invariant sets may appear and be of importance for the asymptotic behavior of typical initial conditions. These not only organize the autonomous dynamics but also allow one to understand how the dynamics behave under small perturbations of various types.

If (\ref{eq:ode}) is equivariant under the action of a compact Lie group $\Gamma$ acting orthogonally on $\R^n$, then there is an extensive literature considering many heteroclinic networks with the remarkable property that they are persistent or robust under perturbations of $f$ that respect the symmetries $\Gamma$: see for example  the work of Krupa and Melbourne \cite{Kru97,KruMel95a,KruMel04}.

There are several possible ways to understand a heteroclinic network in a graph-theoretic manner. Note that in many cases the directed graph (digraph) is referred to simply as a graph. Our approach gives the minimal possible graph one could naturally associate with a heteroclinic network: we say there is an edge between vertices if there is at least one connection between the corresponding nodes. 
Another approach is to have an edge for every connection between nodes: this is appropriate for many cases investigated in the literature (e.g.\ \cite{KruMel95a}), but typically results in infinite graphs for the networks we consider here. Yet another choice could be to have an edge for each connected component of the set of connections between nodes. In some instances this may also result in a more complicated graph.

Many papers in the literature consider heteroclinic networks as unions of heteroclinic cycles (e.g.\ Hoyle \cite{Hoy06}) and in cases where there are one-dimensional unstable manifolds this is highly appropriate. In this paper we take a different view however -- we consider the heteroclinic network as the fundamental definition and show in Lemma~\ref{lem:cycles} that a network is a union of cycles, or cycles are cyclic subsets of a network.

This paper is structured as follows: in Section \ref{sec:as-stab} we discuss the relation between directed graphs and heteroclinic cycles and networks, introducing the properties of complete, almost complete, equable and exclusive nodes and networks and give examples of these. We also recall the definition of a clean network. Section \ref{sec:completion} shows in Theorem~\ref{thm:completion} that the simplex method of \cite{AshPos13} allows one to construct realizations of a large class of directed graphs as an almost complete and equable heteroclinic network that is part of a larger closed network. We conjecture, in Section~\ref{sec:discussion}, that this result can be strengthened by (a) widening the class of directed graphs and (b) providing a stronger result -- that the embedding network is not just closed but clean.

In Section~\ref{sec:example} we discuss a number of examples that clarify and illustrate these results and concepts. Section~\ref{sec:application} presents a simple stochastic model of randomly perturbed dynamics on a heteroclinic network. For this model, typical trajectories will only explore an almost complete and equable subnetwork. In this sense, the almost complete and equable subnetwork can be seen as an optimal realization of the network with added noise.

Section~\ref{sec:discussion} concludes with a discussion.

\section{Heteroclinic networks and directed graphs}\label{sec:as-stab}


Given the close relation between heteroclinic networks and directed graphs (see for instance \cite{AshPos13} or \cite{Fie17}), we begin with a section that establishes terminology and notation that allows for an easy transition between the two. A substantial part of this section is not original. We include it because we believe that it is useful for most readers to have the relevant concepts framed in a convenient way.

We denote by $\alpha(x)$ (resp. $\omega(x)$) the usual limit set of the trajectory through $x$ as $t\rightarrow -\infty$ (resp. $\infty$). For a heteroclinic network, there is a natural graph structure between nodes representing the equilibria, such that edges in the graph correspond to connections between equilibria in the network. However, the correspondence is more subtle than one might suppose as the set of connections may consist of many, or even a continuum of trajectories \cite{AshCho98,AshFie99}.

We define the unstable and stable sets of an equilibrium $\xi$ as usual
$$
W^u(\xi)=\{x\in\R^n~:~\alpha(x)=\xi\}, ~~~W^s(\xi)=\{x\in\R^n~:~\omega(x)=\xi\}
$$
and note that for hyperbolic equilibria $\xi$ these are flow-invariant manifolds with dimension 
corresponding to the dimensions of unstable and stable eigen\-spaces of $\xi$. Suppose we have a finite collection of hyperbolic equilibria 
$$
N=\{\xi_1, \hdots, \xi_m\}
$$
for (\ref{eq:ode}).  We define the {\em full set of connections from $\xi_i$ to $\xi_j$} ($\xi_i, \xi_j \in N$) as
$$
C_{ij}=W^u(\xi_i)\cap W^s(\xi_j).
$$
This is a flow-invariant (possibly empty) set: if $i\neq j$ we refer to each trajectory in $C_{ij}$ as a \emph{connection from $\xi_i$ to $\xi_j$}. We include cases where $C_{ij}$ is  continuum of connections \cite{AshCho98}. In the case $i=j$ we call a connection {\em homoclinic}, otherwise we say it is {\em heteroclinic}.\footnote{In equivariant systems, if $\xi_i\neq\xi_j$ but they are in the same group orbit, then some authors consider the connection homoclinic. We do not make this distinction until Section \ref{sec:example}.}

The {\em full set of connections between equilibria in $N$} is defined
$$
C(N)=\bigcup_{i\neq j} C_{ij}.
$$
In what follows we use the notation $C(.)$ to describe the connections associated with the object in brackets.

{\bf We make a standing assumption} that there are no homoclinic connections, i.e.\ we assume that $C_{ii}=\{\xi_i\}$ for all $i$.

Many references in the literature use the following definitions:
A {\em heteroclinic cycle} is a union of finitely many hyperbolic equilibria connected by trajectories in a cyclic way. A {\em heteroclinic network} is a connected union of finitely many heteroclinic cycles.

When studying heteroclinic networks from directed graphs, another definition may be convenient. The relation between the two is clarified in Lemma~\ref{lem:cycles}.
Recall that an invariant set $\Sigma$ is {\em indecomposable} (cf \cite{AshFie99}) for the dynamics of (\ref{eq:ode}) if for every $\epsilon>0$ and pair of points $a,b\in \Sigma$ there is a directed $\epsilon$-chain from $a$ to $b$ within $\Sigma$, where an $\epsilon$-chain is a sequence of points $\{x_k\}_{k=1}^n$ in $\Sigma$ and times $\{t_k>1\}_{k=1}^{n-1}$ such that $x_1=a$, $x_n=b$ and $|\varphi(t_k,x_k)-x_{k+1}|<\epsilon$ for $k=1,\ldots,n-1$. By $N(\Sigma)$ we denote the set of equilibria in $\Sigma$ which we will assume is finite. The following definition is a special case of \cite[Definition 2.26]{AshFie99} where the nodal set is $N(\Sigma)$ and the ``depth'' \cite[Definition 2.22]{AshFie99} is one because all trajectories are either in the nodal set, or limit to the nodal set.\footnote{More general heteroclinic networks, in the sense of \cite{AshFie99}, can have higher depth connections in that they can contain trajectories that limit to connections.}

\begin{definition}\label{def:alternative}
We say $\Sigma$ is a {\em heteroclinic network} between equilibria $N(\Sigma)$ if it is an indecomposable flow-invariant set such that $$
N(\Sigma)\subset \Sigma\subset N(\Sigma) \cup C(N(\Sigma)).
$$
\end{definition}

We refer to the equilibria $N(\Sigma)=\{\xi_1,\ldots,\xi_k\}$ as the {\em nodes} of the network and define
$$
C_{ij}(\Sigma)=C_{ij}\cap \Sigma.
$$
as the {\em connection from $\xi_i$ to $\xi_j$ within the network} $\Sigma$. Note that there may be many connecting trajectories between $\xi_i$ and $\xi_j$ in $\Sigma$ and also some that we do not include in a particular $\Sigma$. Note the decomposition
$$
\Sigma=N(\Sigma)\cup C(\Sigma)
$$
is a disjoint union, where $C(\Sigma):=C(N(\Sigma))\cap \Sigma$ denotes the {\em connections within the network}.

To structure our discussion of graphs associated with heteroclinic networks, we 
use the following notation to go between these concepts:
\begin{itemize}
\item  $G(\Sigma)$ to denote the graph related to a given heteroclinic network $\Sigma$;
\item  $\Sigma_G$ to denote a heteroclinic network associated to a given graph $G$ (this may not be unique).
\end{itemize}
We start with graphs.
 Associated with any heteroclinic network $\Sigma$ there is a digraph $G(\Sigma)=(V,E)$ with vertices $V=\{v_1,\ldots,v_k\}$, where $v_j$ corresponds to node $\xi_j\in N(\Sigma)$, and the set $E$ of {\em directed edges}, where $[v_i\to v_j] \in E$ corresponds to $C_{ij}(\Sigma) \neq \emptyset$ with $i\neq j$. For a given $G(\Sigma)$ we write $N(v)$ to denote the equilibrium corresponding to vertex $v$. Note that $C_{ij}(\Sigma)$ is the full set of connections from $\xi_i$ to $\xi_j$ in $\Sigma$ corresponding to the edge $[v_i\to v_j]$.

As usual\footnote{There are several good references for graph theory, we refer the reader to for example \cite{Die05,Fou92}.}, we say that a {\em cycle} is a sequence of vertices and edges $\{v_1,[v_1\to v_2], v_2,[v_2\to v_3], v_3,\ldots,[v_{m-1}\to v_m],v_m\}$ such that $v_1=v_m$ and all other vertices are distinct. A cycle with $m$ edges is called an {\em $m$-cycle}. A $3$-cycle is also called a {\em triangle}. In the context of digraphs, we reserve the term $m$-cycle for those that are transitive, that is, an oriented circuit through all the vertices; we use triangle for both the transitive and the non-transitive case.
Recall that $G$ is {\em transitive} if for any two distinct vertices $v_i$, $v_j$ there is a directed path from $v_i$ to $v_j$ within $G$.

\begin{definition}
Suppose that $G=(V,E)$ is a digraph.
\begin{itemize}
\item $G$ is an {\em $\Delta$-clique} if it is a triangle that is not transitive (see Figure~\ref{fig:graph_triangles}).
\item Let $V'=\{w,v_1,\ldots,v_k\}$ be the subset of all the distinct vertices of $V$ that $w$ connects to. If the only edges of the graph induced on $V'$ have the form $[w\to v_j]$ for $j=1,\ldots,k$, then we say $w$ is a \emph{splitting vertex of order $k\geq 2$}.
\end{itemize}
\end{definition}

\begin{figure}[!htb]
 \centerline{\includegraphics[width=0.8\textwidth]{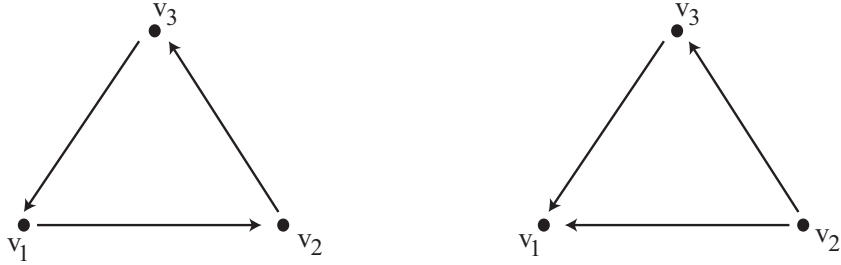}}
 \caption{Two triangles: a $3$-cycle (left) and a $\Delta$-clique (right).\label{fig:graph_triangles}}
 \end{figure}

The use of $\Delta$ in $\Delta$-clique should not be mistaken for the maximum degree, usually denoted by this symbol in graph theory. The symbol $\Delta$ in the present context has a visual association with the dynamics involved.

Note that $w$ being a splitting vertex of order $k$ is a somewhat stronger assumption than simply saying $w$ has out-degree $k:=\#\{j~:~[w \to v_j]\in E\}$ since it also makes assumptions on nearby edges. More precisely:

\begin{lemma}
Suppose $G=(V,E)$ is a transitive digraph. Consider a vertex $w$ and all $v_1,\ldots,v_k$ such that there are edges $[w\to v_j]\in E$. Then $w$ is a splitting vertex for $G=(V,E)$ if and only if the digraph induced on $V'=\{w,v_1,\ldots,v_k\}$ has no $\Delta$-clique or $2$-cycle. 
\end{lemma}

\proof
Suppose that $w$ and $V'$ are as above and there is at least one additional edge in the graph $G'=(V',E')$ induced on $V'$ to those required for a splitting vertex. Then either there is an edge $[v_j\to w]$ and hence there is a $2$-cycle, or there is an edge $[v_j\to v_l]$ and hence there is a $\Delta$-clique on $\{w,v_j,v_l\}$.
\qed
\medskip

We say $\xi$ is a {\em splitting node} for $\Sigma$ if the corresponding vertex is a splitting vertex for $G(\Sigma)$.

\bigbreak

Now suppose $G(\Sigma)=(V,E)$ and consider a subset $N'\subset N(\Sigma)$ of nodes. The induced subgraph $G'=(V',E')$ consists of all edges in $G(\Sigma)$ between vertices in $N'$. This can be used to construct an invariant set
$$
\Sigma_{G'}=\left(\bigcup_{[v_i\to v_j]\in E'} C_{ij}(\Sigma)\right) \cup\left(\bigcup_{v\in V'} N(v)\right)
$$
however, there is no guarantee that $\Sigma_{G'}$ is necessarily transitive or even connected. As an illustration, consider the Kirk and Silber graph (see Figure~\ref{Fig:KScomplete} (left)). If $N'=\{v_3,v_4\}$, then $E'=\emptyset$ and $\Sigma_{G'}$ is not connected. If $N'=\{v_2,v_3,v_4\}$, then $E'=\{[v_2 \to v_3],[v_2 \to v_4]\}$ and $\Sigma_{G'}$ is connected but not transitive.

We say $\Sigma$ is a {\em heteroclinic cycle} if $G(\Sigma)$ is a $k$-cycle for some $k\geq 2$. Note that in such a case, the invariant set $\Sigma$ is not necessarily a topological circle because it may contain multiple connections between two nodes and still be a cycle in our definition. Even worse, it is possible that connections may accumulate on each other away from the equilibria (if this is not the case, then the nodes are exclusive in a sense we define later).  
We now give a lemma that characterizes the relation between heteroclinic networks and digraphs.

\begin{lemma}
For any heteroclinic network $\Sigma$ on $N$, the graph $G(\Sigma)$ is a transitive digraph between vertices $N(\Sigma)$. For any transitive subgraph $H=(V_H,E_H)\subset G(\Sigma)$ there is a heteroclinic network $\Sigma_H\subset \Sigma$ such that $G(\Sigma_H)=H$.
\end{lemma}

\proof
Transitivity of the graph $G(\Sigma)$ follows from  $\Sigma$ being indecomposable. The network $\Sigma_H\subset \Sigma$ can be found by taking the union of equilibria in $V_H$ and connections corresponding to $E_H$:
$$
\Sigma_H=\left(\bigcup_{[v_i\to v_j]\in E_H} C_{ij}(\Sigma)\right) \cup\left(\bigcup_{v\in V_H} N(v)\right).
$$
\qed 
\medskip

This can be used to show the following result, which -- as mentioned above -- is often used as part of the definition of a heteroclinic network.

\begin{lemma}
A heteroclinic network (according to Definition~\ref{def:alternative}) is a connected union of heteroclinic cycles.
\label{lem:cycles}
\end{lemma}

\proof 
Consider a decomposition of a transitive graph $G(\Sigma)$ into a finite union of cycles $G_1,\ldots,G_k$. Each of the $\Sigma_{G_j}\subset \Sigma$ is a heteroclinic cycle but the union of cycles contains $\Sigma$, that is,
$$
\Sigma\subset \bigcup_{j=1}^{k} \Sigma_{G_j}
$$
as it contains all connections and nodes within $\Sigma$. Hence $\Sigma$ is precisely this union.
\qed
\medskip

The minimum length cycles in $\Sigma$ are of interest: we say that the heteroclinic network $\Sigma$ {\em contains a $k$-cycle} for some $k> 1$ if $G(\Sigma)$ contains a $k$-cycle. Note that the decomposition of a heteroclinic network into cycles is usually not unique.  The proof of Lemma~\ref{lem:cycles} implicitly uses a decomposition into cycles that have length equal to the minimum length cycle that returns to any edge but equally there may be a decomposition using longer cycles. Our standing assumption means that $G(\Sigma)$ only contains $k$-cycles for $k\geq 2$.

\subsection{Properties of nodes of heteroclinic networks}

We now consider some properties of connections from nodes within heteroclinic networks.  

\begin{definition}\label{def:nodes}
Suppose that $\Sigma$ is a heteroclinic network and $\xi_i$ a node in that network. We define the following:
\begin{itemize}
\item
$\xi_i$ is \emph{complete in $\Sigma$} if $W^u(\xi_i) \subset \Sigma$ (see Figure~\ref{complete-node-figure}).
\item
$\xi_i$ is \emph{almost complete in $\Sigma$} if $W^u(\xi_i) \setminus \Sigma$ is of measure zero (with respect to Lebesgue measure for any volume form on $W^u(\xi_i)$).
\item
$\xi_i$ is \emph{equable in $\Sigma$} if $C_{ij}$ is a union of manifolds of the same dimension and this dimension $\dim(C_{ij}(\Sigma))$ is equal for all $j$ with $C_{ij}(\Sigma) \neq \emptyset$ (see Figure~\ref{splitting-node-figure}).
\item 
$\xi_i$ is \emph{exclusive in $\Sigma$} (see Figure~\ref{splitting-node-figure})
 if for all $j$ where $C_{ij}(\Sigma)$ is non-empty we have
\begin{equation*}
\overline{C_{ij}(\Sigma)}\cap N(\Sigma) = \{\xi_i,\xi_j\}.
\end{equation*}
\item
$\Sigma$ is a {\em complete}/{\em almost complete}/{\em equable}/{\em exclusive} network if all nodes are respectively complete/ almost complete/ equable/exclusive.
\item
$\Sigma$ is called {\em clean} \cite[Definition 1.3]{Fie17} if it is compact and complete.
\end{itemize}
\end{definition}

We note that the graph of a complete network is not necessarily a complete graph (where every pair of vertices is directly connected by an edge). The network in Figure~\ref{Fig:KScomplete} (right) is complete but the corresponding graph is not ($\xi_3$ and $\xi_4$ are not directly connected, for instance). 

If a node $\xi_i$ is not exclusive, then there exist connections in $C_{ij}$ that are arbitrarily close to a node in the network other than $\xi_i$ and $\xi_j$. Note also that an equable network may have connections of different dimensions. We comment on the effects of equability on the dynamics in Section \ref{sec:application}.

\begin{figure}[!htb]
 \centerline{\includegraphics[width=0.8\textwidth]{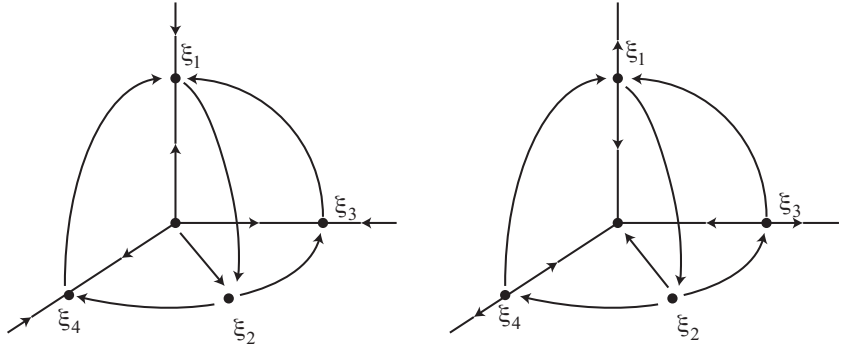}}
 \caption{Let $\Sigma$ be the network with nodes $\xi_i$, $i=1,2,3,4$ and connections between them in $\R^3$. The node $\xi_2$ on the left is complete as the origin is unstable and the 1-dimensional unstable manifold of $\xi_2$ is contained in $\Sigma$. The node $\xi_2$ on the right is not complete as some points in $W^u(\xi_2)$ move away from $\Sigma$.\label{complete-node-figure}}
 \end{figure} 

\begin{figure}[!htb]
 \centerline{\includegraphics[width=0.5\textwidth]{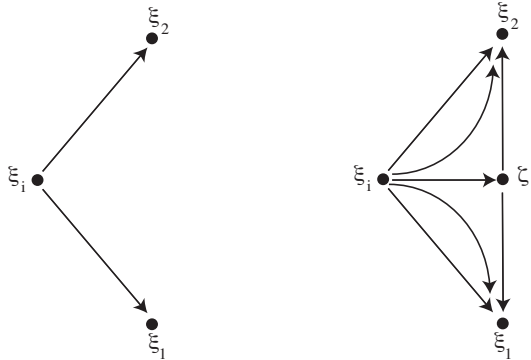}}
 \caption{The splitting node $\xi_i$ (left) is made complete by inserting an additional node $\zeta$ and additional connections (right).The node $\xi_i$ is not equable on the right as the dimension of the connections from $\xi_i$ can be either 1, $[\xi_i\to \zeta]$, or 2, $[\xi_i \to \xi_j]$, $j=1,2$. 
On the left $\xi_i$ is exclusive but on the right it is not: $\zeta \in \overline{C_{ij}}$, $j=1,2$. Note also that on the right $\xi_i$ is no longer a splitting node.
 \label{splitting-node-figure}}
 \end{figure} 

The well-known network of Kirk and Silber \cite{KirSil94} provides an example of an equable network that is not complete, we comment on this in Subsection \ref{subsec:KS}. Other authors have implicitly noted the importance of graph structures such as $\Delta$-cliques for properties of clean heteroclinic networks \cite{Bic18,Fie17}. Note that a clean network need not be equable: we give an example for this in Subsection \ref{subsec:B3B3C4}. 

For a given set of equilibria $N$ it is not necessary that $C(N) \cup N$ is complete or even closed -- this can be for a variety of reasons. Although it may not be true that for a given $N$ there is $N'$ containing $N$ such that $N' \cup C(N')$ is complete, in Section~\ref{sec:completion}, we find constructions such that $N'\cup C(N')$ is at least closed.

The following result highlights that a splitting node $\xi_i$ in a complete and equable network is either very simple and the splitting is of order $2$, or it is not exclusive -- the closure of $C_{ij}$ contains a node $\xi_k$ that is neither $\xi_i$ nor $\xi_j$. If this is the case, then there will be some $\ell\not\in\{i,j\}$ such that $C_{i\ell}\neq\emptyset$.

\begin{lemma}
\label{lem:completenode}
Suppose that the node $\xi_i$ is complete in $\Sigma$. Then $\xi_i$ is almost complete in $\Sigma$. If in addition $\xi_i$ is exclusive, equable and a splitting node of order $k\geq 2$, then $k=2$ and $\dim(W^u(\xi_i))=1$.
\end{lemma}

\proof
Suppose $\dim(W^u(\xi_i))=d$. If $\xi_i$ is complete in $\Sigma$, then $\ell(W^u(\xi_i)\setminus \Sigma)=\ell(\emptyset)=0$ where $\ell$ denotes $d$-dimensional Lebesgue measure on $W^u(\xi_i)$ and so $\xi_i$ is almost complete in $\Sigma$. 

For the second part, pick some small $\delta>0$ such that $S:=\{x\in W^u(\xi_i)~:~|x-\xi_i|=\delta\}$ is diffeomorphic to a $(d-1)$-sphere. Note that the $(d-1)$-sphere is connected for $d\geq 2$ and has two components for $d=1$. If $\xi_i$ is also an equable splitting node of order $k$, then all connections from $\xi_i$ must have the same dimension. If $\xi_i$ is exclusive then $\overline{C_{ij}}\cap S$ will not intersect $\overline{C_{im}}$ for $m \in \{1,\ldots,k\}$, $m\neq j$: hence there is a partition of $S$ into $k$ closed disjoint sets. This is only possible for $k=2$ and $d=1$.
\qed
\medbreak

If $\Sigma$ is a complete heteroclinic network, then by Lemma~\ref{lem:completenode} it is almost complete: moreover, in such a case it is maximal in the sense that
$$
\Sigma=\bigcup_{\xi\in N(\Sigma)} W^u(\xi) \;\; \mbox{ and  } \;\; 
C_{ij}(\Sigma)=C_{ij}.
$$
There is also a partition of $W^u(\xi)$ into a union of connections from $\xi$. Moreover, from the proof it is easy to see that if $\xi_i$ is a splitting node and $W^u(\xi_i)$ is at least 2-dimensional, then $\xi_i$ is not complete. If $W^u(\xi_i)$ is 2-dimensional and $\xi_i$ is complete, then it is not a splitting node.


Note that asymptotic stability of a compact heteroclinic network implies that it is clean  \cite{Fie17}. The opposite is not true: heteroclinic objects can lose asymptotic stability through resonance bifurcations. Take the well-known Guckenheimer-Holmes cycle in \cite{GucHol88} for instance, which is complete, even clean, but unstable when condition (c) in \cite[Lemma 3]{GucHol88} is broken. By contrast, a transverse bifurcation involves a sign change for some eigenvalue (or its real part) and thus affects the completeness of a cycle/network.

It follows from Definition \ref{def:nodes} that if a node $\xi_i$ is such that $W^u(\xi_i)=1$, then $\xi_i$ is equable; such a $\xi_i$ is also exclusive if the nodes are equilibria and the network is of depth one in the sense of \cite{AshFie99}.

If $G$ has a $\Delta$-clique, then it does not follow that there is a non-exclusive or a non-equable node in $\Sigma_G$. Take a $B_2^+$ cycle (see Subsection \ref{subsec:B2B2}) between equilibria $\xi_a$ and $\xi_b$ for example. Add a fourth space dimension with an equilibrium $\xi$ on the extra axis, such that $\dim(W^u(\xi))=1$ and there are connections from $\xi$ to $\xi_a$ and $\xi_b$. Then the three equilibria form a $\Delta$-clique, but $\xi$ is exclusive and equable. This can be embedded in a heteroclinic network in a higher dimensional space.

\section{Realization as almost complete equable heteroclinic networks}\label{sec:completion}

The problem of realizing abstract digraphs as heteroclinic networks was raised in \cite{AshPos13,AshPos16a} and \cite{Fie15}, and several methods have been proposed. Suppose $G=(V,E)$ is an arbitrary transitive digraph. We say the dynamics of (\ref{eq:ode}) {\em realizes} $G$ as the heteroclinic network $\Sigma$ if there is a choice of $f$ and $\Sigma$ such that $G(\Sigma)=G$. Without loss of generality we can choose $\Sigma$ to be the maximal choice, i.e.
$$
\Sigma=\left(\bigcup_{[v_i\to v_j]\in E} C_{ij}(\Sigma) \right) \cup\left(\bigcup_{v\in V} N(v)\right).
$$
In \cite{AshPos13}, two methods are presented to show that, under minimal assumptions, a digraph $G$ can be realized as a heteroclinic network. The simplex method embeds the graph in a simplex by placing the nodes on the coordinate axes. This method realizes the graph provided it has neither $1$- nor $2$-cycles.
The cylinder method places the nodes along one coordinate axis and realizes any graph provided it has no $1$-cycles.

In this section we show that the simplex construction, for a certain choice of parameters, gives a realization that is an almost complete, equable subnetwork of a closed heteroclinic network, and this is robust under certain equivariant perturbations. We state and prove this as Theorem~\ref{thm:completion} and later, in Subsection \ref{subsec:KS} give an example
that elucidates the result and method of proof. 

According to \cite[Proposition 1]{AshPos13} any graph $G$ without 1- and 2-cycles can be realized as a heteroclinic network $\Sigma=C(N) \cup N$ on a set of equilibria $N$. The resulting vector field on $\R^n$, $n=\# N$, is $\Z_2^n$-equivariant and yields an equilibrium on each coordinate axis and connections in coordinate planes. Theorem \ref{thm:completion} shows that under the additional hypothesis that there are no $\Delta$-cliques and with an appropriate choice of parameters, this can be done in such a way that $\Sigma$ is an almost complete, equable subnetwork of a closed network $\Sigma'$. Although the vector field is as in \cite[Proposition 1]{AshPos13} our method of proof involves the construction of Lyapunov-type functions that use the additional hypotheses.

\begin{theorem}\label{thm:completion} 
Let $G$ be a transitive directed graph on $n$ vertices with no 1-cycles, 2-cycles or $\Delta$-cliques. Then there exists a $\Z_2^n$-equivariant vector field $f$ on $\R^n$ that realizes $G$ as a network $\Sigma(N)$ between nodes $N=\{\xi_1,\ldots,\xi_n\}$. This realization is robust to $\Z_2^n$-equivariant perturbations.
The vector field can be chosen such that there is an additional set of nodes $N'$  and a closed heteroclinic network $\Sigma'$ between $N\cup N'$ such that $\Sigma$ is an almost complete, equable subnetwork of $\Sigma'$.
\end{theorem}

\begin{proof}
For $j=1,\ldots,n$ we define the smooth vector field on $\R^n$
\begin{equation}
\dot{x}_j = f_j(x):= x_j F_j(x)
\label{eq:system}
\end{equation}
where
\begin{equation}
F_j(x):=1+\sum_{i} [(\epsilon +\eta)
A_{ij}-\eta(1-\delta_{ij})-1]x_i^2.
\label{eq:defFj}
\end{equation}
We set $A_{ij}=1$ if $G$ prescribes a connection from $\xi_i$ to $\xi_j$, and $A_{ij}=0$ otherwise, while $\delta_{ij}$ is the Kronecker symbol and the constants $\epsilon, \eta$ satisfy $0< \epsilon <1$ and $\eta >0$. 

Since the vector field \eqref{eq:system} satisfies the hypotheses of \cite{AshPos13}, only the last statement requires proof.

There are equilibria at $\xi_j$ corresponding to the unit basis for $\R^n$. Let $N=\{\xi_j\}$ denote these equilibria of (\ref{eq:system}) on the coordinate axes. 

Lack of 1- and 2-cycles can be expressed as
$A_{ii}=0$, $A_{ij}A_{ji}=0$
for all $i$ and $j$, while lack of $\Delta$-cliques means that
$A_{ij}A_{jk}=1$ implies $A_{ik}=0$ for any $i,j,k$. We write
$$
O_j:=\{k\in\{1,\ldots,n\}~:~A_{jk}=1\}
$$
for the non-empty set of indices corresponding to the outgoing directions from $\xi_j$.
The proof proceeds in the following steps.
\paragraph{Step 1 -- existence of an absorbing region for the dynamics:}

We write $R:=|x|^2=\sum_j x_j^2$ and calculate
\begin{eqnarray*}
	\frac{\textnormal{d}}{\textnormal{d}t}R&=& \sum_j 2x^2_jF_j(x)\\
	&=&\sum_j 2x_j^2 \left[1+\sum_i [(\epsilon+\eta) A_{ij}-\eta(1-\delta_{ij})-1]x_i^2 \right]\\
	&=&2R-2R^2+2\sum_{i,j}\left[(\epsilon+\eta) A_{ij}-\eta(1-\delta_{ij})]x_i^2x_j^2 \right].
\end{eqnarray*}
 But note that 
$$
-\eta\leq(\epsilon+\eta) A_{ij}-\eta(1-\delta_{ij})\leq \epsilon
$$
and so
$$
-\eta R^2\leq \sum_{i,j} [(\epsilon+\eta) A_{ij}-\eta(1-\delta_{ij})]x_i^2x_j^2\leq \epsilon R^2,
$$
which implies
$$
2R(1-R-\eta R)\leq \frac{\textnormal{d}}{\textnormal{d}t}R \leq 2R(1-R+\epsilon R)
.$$
This means that there is an absorbing region $R_0<R<R_1$ where
$$
R_0:=\frac{1}{1+\eta} \leq R \leq R_1:=\frac{1}{1-\epsilon}.
$$
Therefore, for any $\eta>0$ and $0<\epsilon<1$ there is an absorbing spherical annulus
$$
S:=\left\{x~:~\frac{1}{1+\eta} \leq |x|^2 \leq \frac{1}{1-\epsilon}\right\}.
$$
\bigbreak
If we fix $j$ and define the invariant subspace
$$
\Omega_j:=\{x~:~x_k=0 ~\mbox{ if }k\not \in O_j\}
$$
then $\xi_j$ has an unstable manifold contained within the invariant subspace
\begin{equation}
Q_j:=\Omega_j \oplus \langle \xi_j\rangle.
\label{eq:defQj}
\end{equation}

\paragraph{Step 2 -- $\Omega_j$ attracts almost all initial conditions in $Q_j$:}
In fact, every trajectory in $Q_j$ that is not in the one dimensional subspace spanned by  $\xi_j$ limits to $\Omega_j$. We define a function $\Phi_j: Q_j \rightarrow \R$ by
 $$
\tan \Phi_j:= \frac{x_j^2}{\sum_{i\in O_j} x_i^2}
$$
and note that for any $x\in Q_j$ and $i\in O_j$ we have
\begin{eqnarray*}
	\dot{x}_j&=&x_j\left[1+\sum_{k\in O_j}[-\eta-1]x_k^2-x_j^2 \right]\\
	\dot{x}_i&=&x_i\left[1+[\epsilon -1]x_j^2+\sum_{i\neq k\in O_j}[-\eta-1]x_k^2-x_i^2 \right].
\end{eqnarray*}

Note that
$$
\frac{\textnormal{d}}{\textnormal{d}t}\left[ \tan \Phi_j\right] = \frac{\textnormal{d}}{\textnormal{d}t}\left[\frac{x_j^2}{\sum_{i\in O_j} x_i^2}\right]
=(1+\tan^2 \Phi_j) \frac{\textnormal{d}}{\textnormal{d}t}\Phi_j,
$$
so that 
$$
\frac{\textnormal{d}}{\textnormal{d}t}\Phi_j = \dfrac{1}{1+\tan^2 \Phi_j}\frac{\textnormal{d}}{\textnormal{d}t}\left[ \tan \Phi_j\right].
$$
Hence we have
\begin{eqnarray*}
	\frac{[\sum_{i\in O_j}x_i^2]^2+x_j^4}{[\sum_{i\in O_j}x_i^2]^2}\frac{\textnormal{d}}{\textnormal{d}t}\Phi_j &=& 2\frac{x_j\dot{x}_j[\sum_{i\in O_j}x_i^2]- x_j^2\sum_{i\in O_j}x_i\dot{x}_i  }{[\sum_{i\in O_j}x_i^2]^2}
\end{eqnarray*}
for all $x\in Q_j\setminus \{0\}$. This means that
\begin{eqnarray*}
	\frac{\textnormal{d}}{\textnormal{d}t}\Phi_j &=& 2\frac{x_j\dot{x}_j[\sum_{i\in O_j}x_i^2]- x_j^2\sum_{i\in O_j}x_i\dot{x}_i }{[\sum_{i\in O_j}x_i^2]^2+x_j^4}\\
	&=& 2\frac{-\eta x_j^2\sum_{k\in O_j}x_k^2[\sum_{i\in O_j}x_i^2]- x_j^2\sum_{i\in O_j}x_i^2[\epsilon x_j^2-\eta\sum_{i\neq k\in O_j}x_k^2]}{[\sum_{i\in O_j}x_i^2]^2+x_j^4}\\
	&=& -2\frac{x_j^2\sum_{i\in O_j}x_i^2(\eta x_i^2+\epsilon x_j^2)}{[\sum_{i\in O_j}x_i^2]^2+x_j^4}.
\end{eqnarray*}
For any $\eta>0$ and $\epsilon>0$ this quantity is clearly finite as long as $x\neq 0$ and non-positive except when $x_i^2=0$ for all $i\in O_j$. Hence $\Phi_j$ decreases monotonically to $0$ for any initial conditions in $Q_j\cap S$ except when $x_i^2=0$ for all $i\in O_j$. This implies that all initial conditions except those on the $x_j$-axis converge to $\Omega_j$.

\paragraph{Step 3 -- the dynamics restricted to $\Omega_j$ is a gradient flow:}
Suppose $x\in \Omega_j$ so that $R=\sum_{i\in O_j}x_i^2$.
Let
$$
V_j=-\frac{1}{2}R+\frac{1}{4}R^2+\frac{1}{4}\eta \sum_{k\in O_j} x_k^2\left [\sum_{l\in O_j,~l\neq k} x_l^2\right]
$$
and note that for any $i\in O_j$ and $x\in \Omega_j$ we have
$$
-\frac{\partial}{\partial x_i} V_j = x_i(1-R) - \eta x_i \sum_{k\neq i}x_k^2= x_j F_j(x).
$$
Hence the flow (\ref{eq:system}) is a gradient flow when restricted to any $\Omega_j$.

To conclude the proof, note that the only minima of $V_j$ on $\Omega_j$ correspond to stable equilibria of the vector field which are at $x=\xi_i$ for each $i\in O_j$. These equilibria are linearly stable, meaning they are quadratic minima for $V_j$ on $\Omega_j$. All other 
stationary points of $V_j$ are quadratically non-degenerate and correspond to saddles or repellers of (\ref{eq:system}) on $\Omega_j$. This means that the flow on $\Omega_j$ is Morse-Smale and robust to perturbations.
We define the separating nodes $N'$ to be the union of all additional stationary points of $V_j$, and we define the heteroclinic network $\Sigma'$ to be the closures of the unstable manifolds of the $\xi_j$.

Note that $Q_j$ contains $W^u(\xi_j)$ and almost any trajectory in $Q_j$ limits to an equilibrium in $\Omega_j$. Hence by including all equilibria in $Q_j$ in $N'$ we ensure that all $W^u(\xi_j)$ consist of connections between equilibria in the network.
More precisely, any initial condition in $Q_j$ that is in 
$$
T_i=\{x\in Q_j~:~ x_i^2>x_k^2 \mbox{ for all }k\in O_j\}
$$
is asymptotic to $\xi_i$. Because $W^u(\xi_j)$ is transverse to the radial direction $\langle \xi_j\rangle$, almost all trajectories in $W^u(\xi_j)$ limit to one of the stable equilibria $\xi_i$, ensuring that $\xi_j$ is equable in $\Sigma$.
\end{proof}

Note that the network $\Sigma'$ is not just closed but clean if all separating nodes $N'$ have unstable manifolds that are entirely contained within $Q_j$ for some $j$. We see in Subsection~\ref{subsec:KS} that this need not be the case, even for a simple but nontrivial network.

\section{Examples}\label{sec:example}

In this section we discuss several examples to illustrate what it means for a network/node to be (in)complete and/or equable. In an equivariant setting, simple heteroclinic cycles\footnote{Cycles are defined as simple in \cite{KruMel04} if the nodes are in different connected components of 1-dimensional fixed-point spaces and the connections are in 2-dimensional fixed-point spaces.} have been classified into types A, B or C by Krupa and Melbourne \cite{KruMel04}. We use their terminology here to indicate the type of a cycle (through the respective letter) and its number of equilibria (as a subscript). The superscript $+/-$ encodes information about the symmetry group that is not relevant for our discussion.

All of our examples are equivariant under the action of some symmetry group $\Gamma$. We identify objects in the same group orbit so that when the graph has vertices $\xi_i$ and $\xi_j$ and an edge $[\xi_i \to \xi_j]$, the network has connections between the corresponding elements in the group orbits $\Gamma.\xi_i$ and $\Gamma.\xi_j$. These connections are symmetric images of one another.

\subsection{The Kirk and Silber/ $(B_3^-,B_3^-)$ network}\label{subsec:KS}

The heteroclinic network of Kirk and Silber \cite{KirSil94} consists of two cycles of type $B_3^-$ with connections (typically viewed as one-dimensional) between equilibria $\xi_1, \xi_2, \xi_3, \xi_4 \in \R^4$. Note that aspects of this network were previously discussed in \cite[Examples 2.10]{Fie17} and \cite[Case I]{Kiretal12}.

It realizes the graph given in Figure~\ref{Fig:KScomplete} (left). The vector field realizing the network robustly has symmetry $\Z^4_2$ where the group acts as multiplication by $-1$ in each coordinate. The node $\xi_2$ is a splitting node and not complete.\footnote{When a connection $[\xi_i \to \xi_j]$ exists, the connection $[\xi_i \to -\xi_j]$ also exists. However, under the identification of objects in the same group orbit, only $\xi_2$ is a splitting node.} The closure of its unstable manifold contains a separating node $\zeta$ (in the plane containing $\xi_3$ and $\xi_4$) and connections $[\xi_2 \to \zeta]$, $[\zeta \to \xi_3]$ and $[\zeta \to \xi_4]$, see Figure~\ref{Fig:KScomplete} (right).  

\begin{figure}[!htb]
 \centerline{\includegraphics[width=0.9\textwidth]{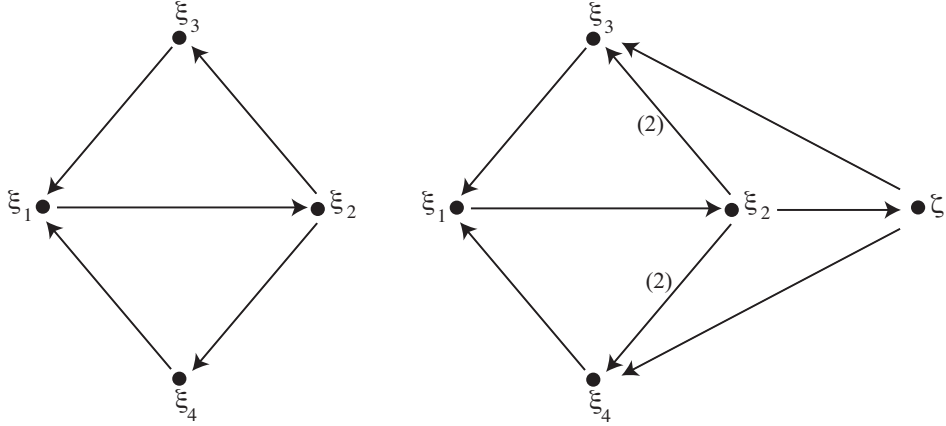}}
 \caption{The $(B_3^-,B_3^-)$ network has an incomplete splitting node at $\xi_2$ (left). Adding a node $\zeta$ and connections $[\xi_2 \to \zeta]$, $[\zeta \to \xi_3]$ and $[\zeta \to \xi_4]$ to the network makes $\xi_2$ complete (right). On the left, $\xi_2$ is equable and exclusive whereas on the right it is not. The numbers in brackets correspond to the dimension of the connection when this is greater than one. \label{Fig:KScomplete}}
\end{figure}

To better illustrate the construction in the proof of Theorem~\ref{thm:completion}, we apply it to this graph/network. With the given connection structure we obtain the system (\ref{eq:system},\ref{eq:defFj}) which can be written
\begin{eqnarray}
	\dot{x}_1&=& x_1[1-|x|^2+\epsilon (x_3^2+x_4^2)-\eta x_2^2]\nonumber\\
	\dot{x}_2&=& x_2[1-|x|^2+\epsilon x_1^2 -\eta (x_3^2 + x_4^2)]\nonumber\\
	\dot{x}_3&=& x_3[1-|x|^2+\epsilon x_2^2 -\eta(x_1^2+x_4^2)]\nonumber\\
	\dot{x}_4&=& x_4[1-|x|^2+\epsilon x_2^2 -\eta(x_1^2+x_3^2)].\label{eq:KSrealise}
\end{eqnarray}
As required this system has four equilibria $\xi_i$ on the unit coordinate axes. We note that
$$
\Omega_1=\{(0,a,0,0)\},~\Omega_2=\{(0,0,a,b)\},~\Omega_3=\Omega_4=\{(a,0,0,0)\}.
$$
and (\ref{eq:defQj}) means that if $x\in Q_2$ then $x=(0,x_2,x_3,x_4)$. Hence, if $x\in Q_2$, we have
\begin{eqnarray*}
	\dot{x}_2&=& x_2[1-|x|^2-\eta(x_3^2+x_4^2)]\\
	\dot{x}_3&=& x_3[1-|x|^2+\epsilon x_2^2 -\eta x_4^2]\\
	\dot{x}_4&=& x_4[1-|x|^2+\epsilon x_2^2 -\eta x_3^2].
\end{eqnarray*}
In fact the only attractors in $Q_2$ are $\xi_3$ and $\xi_4$: consider
$$
\tan \Phi_2:= \frac{x_2^2}{x_3^2+x_4^2},
$$
then
$$
\frac{\textnormal{d}}{\textnormal{d}t}\Phi_2= -2x_2^2\frac{\epsilon(x_3^2+x_4^2)x_2^2+\eta(x_3^4+x_4^4)}{x_2^4+(x_3^2+x_4^2)^2
}
$$
which on $Q_2$ is clearly decreasing to $x_2=0$ unless $x_3^2+x_4^2=0$. Finally, if we define
$$
V_2(0,0,x_3,x_4):
=-R/2+R^2/4+\eta x_3^2x_4^2/2
$$
then on $\Omega_2$ we have
\begin{eqnarray*}
	\dot{x}_3&=&-\frac{\partial V_2}{\partial x_3}\\
	\dot{x}_4&=&-\frac{\partial V_2}{\partial x_4}
\end{eqnarray*}
giving a gradient flow in $\Omega_2$, with almost all trajectories converging to $\xi_3$ and $\xi_4$ (minima of $V_2$).

Restricting the flow to $\Omega_2$ we find a separating node $\zeta = (0,0,x_3,x_4)$ such that $x_3^2=x_4^2=\frac{1}{2+\eta}$. As illustrated in the right panel of Figure~\ref{Fig:KScomplete}, $\zeta$ is a saddle in $\Omega_2$. Its unstable space includes however the direction of $\xi_1$ showing that the unstable manifold of $\zeta$ is not contained in $Q_2$. In this case a clean network can be obtained by also including $W^u(\zeta)$ which means having a two-dimensional connection $[\zeta \to \xi_1]$ in the network.

\subsection{The $(B_3^-,B_3^-,C_4^-)$ network}\label{subsec:B3B3C4}

This network, along with other examples, is discussed in the context of clean networks in \cite[Examples 2.10]{Fie17}. It also appears in Brannath \cite{Bra94} and Castro and Lohse \cite{CasLoh16}. Its graph is given in Figure~\ref{B3B3C4complete} and the simplex method provides a vector field with symmetry $\Z_2^4$ as above that realizes it. The $(B_3^-,B_3^-,C_4^-)$ network has two $\Delta$-cliques: one involving the nodes $\xi_2$, $\xi_3$ and $\xi_4$; the other involving the nodes $\xi_3$, $\xi_1$ and $\xi_4$. Hence, it does not satisfy the hypotheses of Theorem~\ref{thm:completion}. However, the network is clean since the nodes with out-degree greater than 1, $\xi_2$ and $\xi_3$, are complete. Note that $\xi_2$ and $\xi_3$ are not equable: $\dim(\Sigma \cap C_{23})=1 \neq 2 =\dim(\Sigma \cap C_{24})$ and $\dim(\Sigma \cap C_{34})=1 \neq 2 =\dim(\Sigma \cap C_{31})$.

\begin{figure}[!htb]
 \centerline{\includegraphics[width=0.4\textwidth]{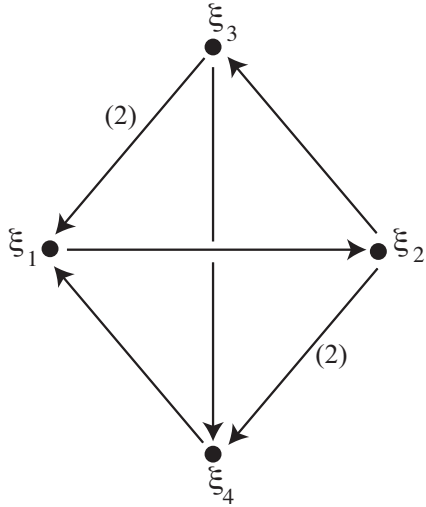}}
 \caption{The $(B_3^-,B_3^-,C_4^-)$ network has no splitting nodes and is clean. There are two $\Delta$-cliques involving the non-equable, but complete, nodes $\xi_2$ and $\xi_3$. The numbers in brackets correspond to the dimension of the connection when this is greater than one.\label{B3B3C4complete}}
 \end{figure}
 
There are infinitely many instances of the $(B_3^-, B_3^-)$ network with one-dimensional connections as equable subnetworks. Their union forms an almost complete, but non-equable subnetwork with the same $(B_3^-, B_3^-)$ graph, but there is no subnetwork of the $(B_3^-, B_3^-, C_4^-)$ network that is both equable and almost complete.

\subsection{The $(B_2^+,B_2^+)$ network}\label{subsec:B2B2}

A network with two cycles of type $B_2^+$ is described in Castro and Lohse \cite{CasLoh14}. Note that even though this object is usually referred to as a heteroclinic network, our Definition~\ref{def:alternative} classifies it as a heteroclinic cycle. In this sense, our definition of heteroclinic cycle is less strict than many definitions in the literature. According to results in \cite{AshPos13}, the cylinder method can be used to provide a vector field in $\R^4$ realizing the corresponding graph. The network is clean and equable. It has no splitting nodes since all connections are between the same two equilibria. 

The vector field supporting the $(B_2^+,B_2^+)$ network has symmetry $\Z_2^3$ where the action of $\Z_2$ is multiplication by $-1$ of each of the last three coordinates of $\R^4$. There is a one-dimensional connection $[\xi_a \to \xi_b]$. The full set of connections $C_{ba}$ consists in three types (distinguished by isotropy) of connections: a one-dimensional connection contained in the $(x_1,x_3)$-plane, another one-dimensional connection in the $(x_1,x_4)$-plane and a two-dimensional connection in the $(x_1,x_3,x_4)$-space. See Figure~\ref{B2B2complete}.

\begin{figure}[!htb]
\centerline{\includegraphics[width=0.8\textwidth]{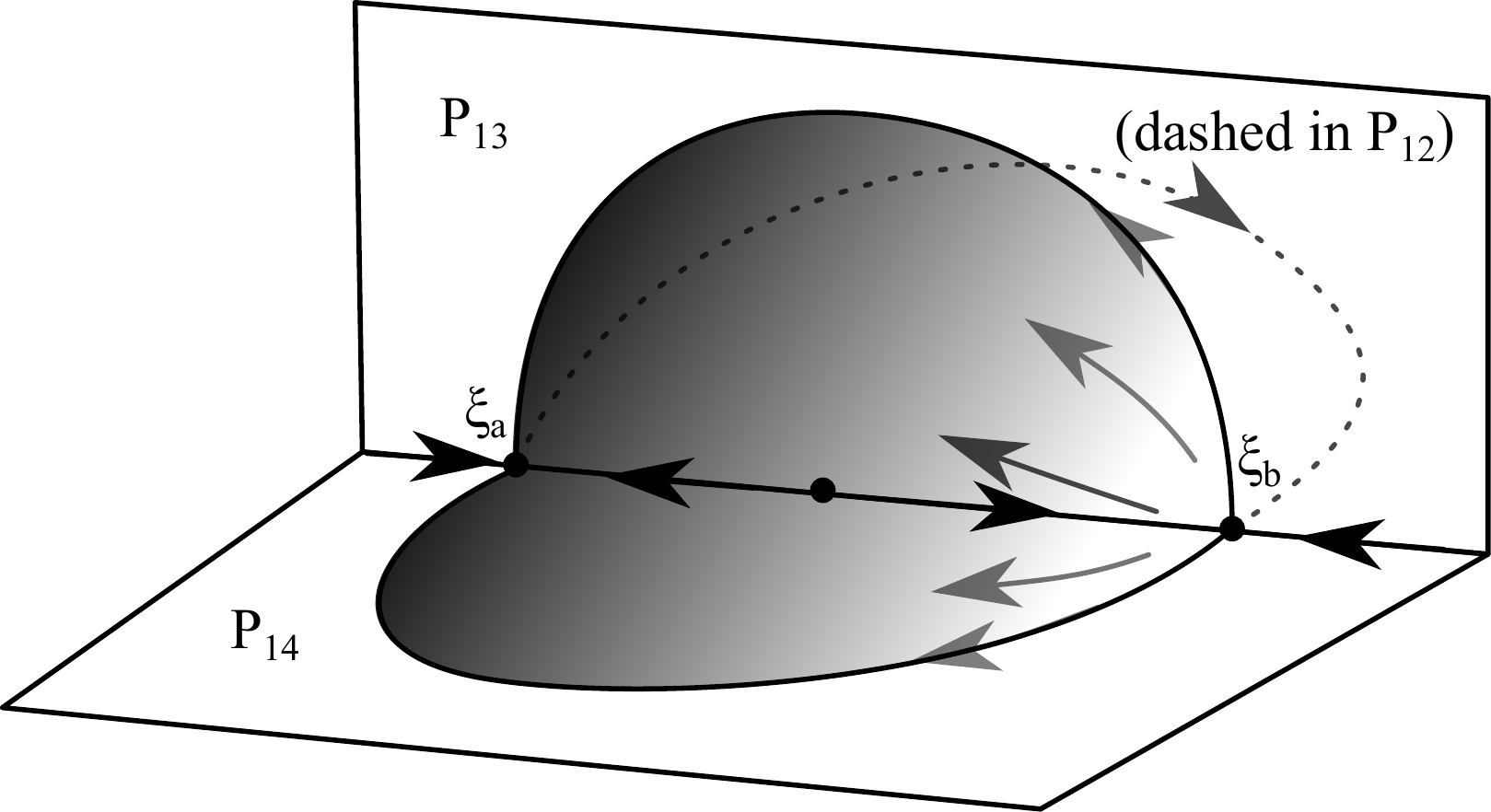}}
 \caption{The $(B_2^+,B_2^+)$ network is clean, equable and exclusive. There are no splitting nodes, the connection from $\xi_a$ to $\xi_b$ is one-dimensional, while the connection back is two-dimensional. The shaded area shows a two-dimensional set of the connections in $C_{ba}$. The connection $C_{ab}$ (dashed) is contained in $P_{12}$.\label{B2B2complete}}
 \end{figure}

Certainly many more examples can be found in the literature. For instance, Kirk et al.\ \cite{Kiretal10} discuss a non-simple network in $\R^4$ with six equilibria and $\Z_2^3$ symmetry that is clean, but not equable. It is obtained by neither the simplex nor the cylinder method and there are nodes where the linearization has complex eigenvalues.

\section{A Markov switching process and almost complete equable networks}

\label{sec:application}

To give some more insight to the importance of almost complete and equable networks, we consider a heteroclinic network $\Sigma=C(N)\cup N$ and define the following idealized (but somewhat natural) discrete-time model of stochastic dynamics on a network. 

For each node $\xi_i\in N$ we consider a probability measure $\rho_i(x)$ that is supported and absolutely continuous on $W^u(\xi_i)$ with respect to a Lebesgue measure, and whose density is non-zero in some neighbourhood of $\xi_i$. We define a one-step discrete-time {\em Markov switching process} $\Xi=\{\xi(n)\in N\cup \{e\}\}_{n\in\Z}$ on $\Sigma$ where $e$ represents an ``escaped" state. We define the switching probability from $\xi_j$ to $\xi_k$ by 
\begin{equation}
\mathcal{P}(\xi(n+1)=\xi_k|\xi(n)=\xi_j)=\rho_j(C_{jk}(\Sigma)).
\label{eq:transitions}
\end{equation}
Note that if a node is not almost complete, then paths of the process can ``leak out" from that node: If we define
\begin{equation*}
\mathcal{P}(\xi(n+1)=e|\xi(n)=\xi_j)=1-\sum_{k}\rho_j(C_{jk}(\Sigma)),
\end{equation*}
then this may be non-zero. Finally, we assume $\mathcal{P}(\xi(n+1)=e|\xi(n)=e)=1$.

The following lemma shows that in cases where this process almost surely does not escape, it explores an almost complete equable subnetwork of $\Sigma$. This subnetwork is obtained by ignoring for each node all lower-dimensional connections (and corresponding nodes) that make it non-equable, e.g.\ $\zeta$ and the connections leading to and from it in Figure \ref{splitting-node-figure}. It is maximal in the sense that it contains all other equable, almost complete subnetworks.

\begin{proposition}
Consider a heteroclinic network $\Sigma$ supporting a Markov switching process $\Xi$. If $\Xi$ starting at any point on $N(\Sigma)$ almost surely avoids escape, then $\Sigma$ is almost complete. Moreover, there is an equable almost complete subnetwork $\Sigma^*$ such that only transitions within $\Sigma^*$ are seen with positive probability.
\label{prop:markov}
\end{proposition}

\proof
Note that by definition, $\Xi$ only gives positive probability to transitions that correspond to positive measure subsets of $W^u(\xi_j)$ for all $j$. 
If $\mathcal{P}(\xi(n+1)=e|\xi(n)=\xi_j)=0$, then $\sum_k\rho_j(C_{jk}(\Sigma))=1$ for all $j$ and so $\Sigma$ is almost complete. Finally, note that (\ref{eq:transitions}) implies that the only connections with non-zero probability of appearing correspond to those with positive measure within $W^u(\xi_i)$ and hence those within some equable almost complete subnetwork.
\qed

We note that Proposition~\ref{prop:markov} is not an equivalence. The converse, i.e.\ that for an equable, almost complete network the Markov process almost surely avoids escape and explores the entire network, will not hold if there is a connecting set $C_{ij}$ of dimension $d$ but zero $d$-dimensional measure.

Consider the stochastic differential equation (SDE)
\begin{equation*}
\textnormal{d}x=f(x)\textnormal{d}t+\alpha \textnormal{d}W_t,
\end{equation*}
where $f$ realizes a given graph as an attracting heteroclinic network for (\ref{eq:ode}), $\alpha>0$ is some small constant and $W_t$ a standard $n$-dimensional Wiener process. As an example, Figure~\ref{FigKSnoise} shows a typical trajectory for the realization of the Kirk-Silber network (\ref{eq:KSrealise}) on $\R^4$ with added noise \cite{ArmStoKir03}. The figure shows a single trajectory exploring almost all directions of exit from the saddle $\xi_2=(0,\pm 1,0,0)$. For smaller noise level $\alpha$, observe that the links become more concentrated around the one-dimensional connections, but still other regions of the manifold are visited with apparent non-zero probability.

Note that a noise-forced heteroclinic system need not behave as a Markov switching process on the last visited node, even in the low noise limit due to the effect of ``lift-off'' \cite{ArmStoKir03,Bak11}.  Nonetheless we do expect the Markov switching process to be a reasonable model for the long-term behaviour of solutions of the SDE in the low noise case if all saddles are ``locally stable'', i.e.\ if the real parts of all expanding eigenvalues at the saddle are smaller in magnitude than the real part of the weakest contracting eigenvalue. This should be valid for the constructions in the proof of Theorem~\ref{thm:completion}, though for other choices of parameters it may no longer be the case.

\begin{figure}[!htb]
 \centerline{\includegraphics[width=\textwidth]{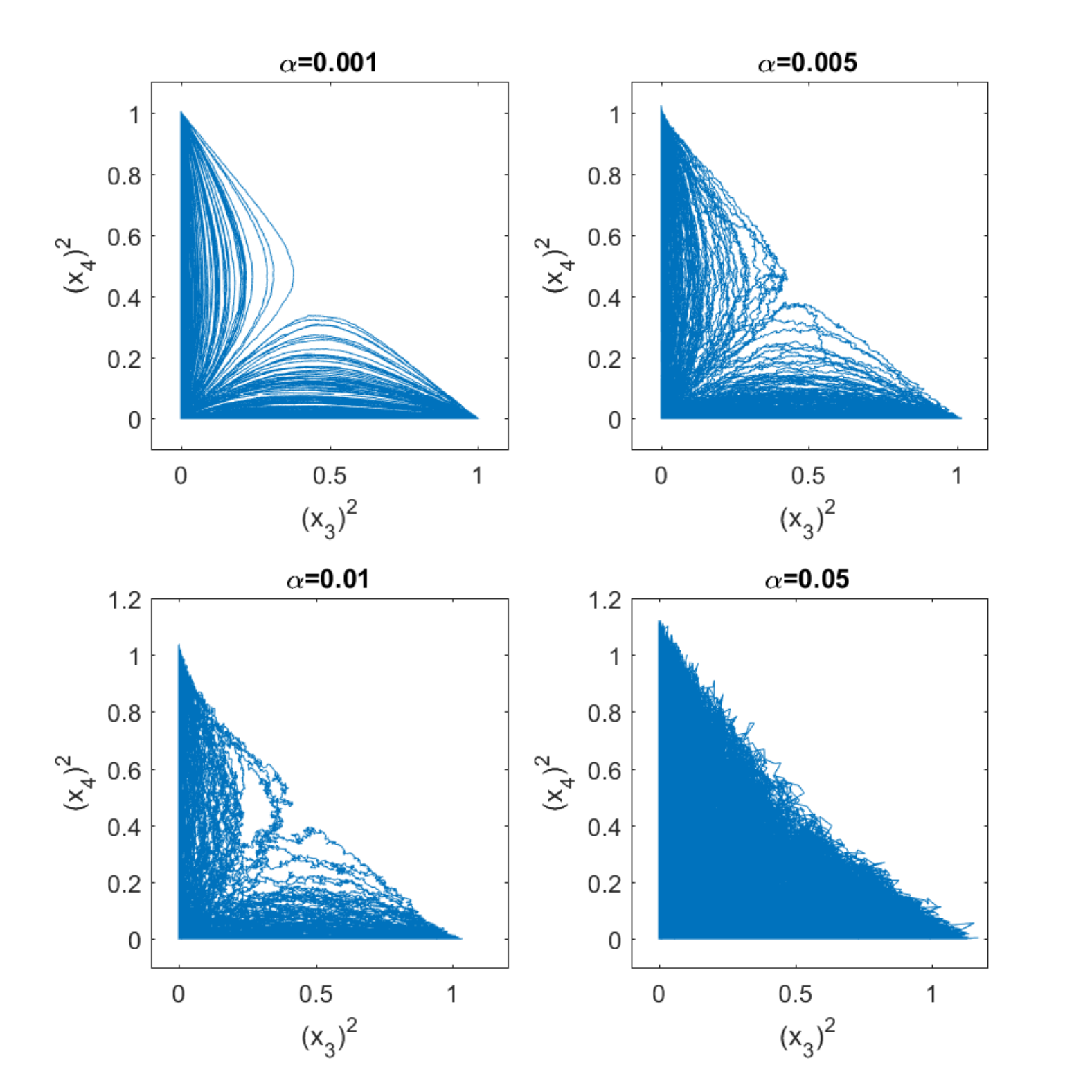}}
 \caption{Trajectories on the unstable manifold of $\xi_2=(0,\pm 1,0,0)$ for the realisation (\ref{eq:KSrealise}) of the Kirk-Silber network with $\epsilon=0.02$, $\eta=0.05$ and increasing noise amplitude $\alpha$. Repeated visits of a single trajectory to
 $x_1^2(t)<0.1$ are shown -- these apparently fill out the 2D unstable manifold of $\xi_2$, and visit arbitrarily closely to the additional node $\zeta$ on the diagonal, as shown in Figure~\ref{Fig:KScomplete}. This simulation uses a Heun integrator with timestep $0.2$. \label{FigKSnoise}}
 \end{figure}

\section{Discussion}\label{sec:discussion}

In summary, we highlight that not only is it possible to realize quite general directed graphs as heteroclinic networks, also these realizations can be maximal in the sense of being almost complete and equable. In addition to the main result Theorem~\ref{thm:completion} and examples in Section~\ref{sec:example}, we present a Markov model and a sense in which almost complete and equable networks can be seen as optimal models of heteroclinic networks perturbed by noise.

While an assumption of no 1-cycles in $G$ is necessary for a robust realization of $G$ as a heteroclinic network, the lack of 2-cycles or $\Delta$-cliques assumed in Theorem~\ref{thm:completion} is presumably not necessary. Indeed, other realization methods  \cite{AshPos13,AshPos16a,Fie15} give robust realizations for $G$ purely on an assumption of no 1-cycles. We conjecture there are parameter choices that give an equivalent result to that in Theorem~\ref{thm:completion} in this more general case. This suggests the following:

\begin{conj}
The conclusion of Theorem~\ref{thm:completion} holds even for directed graphs $G$ that may contain $2$-cycles and $\Delta$-cliques.
\label{conj:completion}
\end{conj}

Explicit constructions are shown as the cylinder method of \cite{AshPos13} and the two layer network \cite{AshPos16a}. These show the existence of networks $\Sigma$ that are equable subnetworks realizing $G$ as long as $G$ has no $1$-cycles, and these realizations can be made robust to certain symmetric perturbations. The problem remains to show that the network is almost complete. Note that the $(A_2^+,A_2^+)$- and $(B_2^+,B_2^+)$-networks are simple and may be created by the cylinder method. However, cycles or networks with more than two equilibria that are generated in this way are not simple, because all equilibria are on the same coordinate axis $L$ -- violating the condition that every connected component of $L \setminus \{0\}$ contains at most one equilibrium.  Finally, the unstable manifolds for the cylinder construction are highly curved and it seems much harder to find suitable Lyapunov-type functions as used in the proof of Theorem~\ref{thm:completion}.


Finally, we remark that the construction in Theorem~\ref{thm:completion} (or a strengthened version Conjecture~\ref{conj:completion}) can presumably be strengthened in the following way: It should be possible to show that under the same (or weakened) hypotheses of Theorem~\ref{thm:completion}, an explicit realization can be chosen such that the embedding network $\Sigma'$ is clean. The main obstruction to showing this is explicitly making the separating nodes in $N'\cap Q_j$ transversely stable to $Q_j$. Although it is clear that this only requires a local change to the transverse stability at all separating nodes, it is still a challenge to explicitly give the construction.

\paragraph{Acknowledgements:}

We thank Chris Bick, Mike Field and Claire Postlethwaite for their insightful questions and comments. The second author was partly supported by CMUP (UID/MAT/00144/2013), funded by the Portuguese Government through the Funda\c{c}\~ao para a Ci\^encia e a Tecnologia (FCT) with national (MEC) and European structural funds through the programs FEDER, under the partnership agreement PT2020. The second and third authors benefitted from DAAD-CRUP funding through ``A\c{c}\~ao Integrada Luso-Alem\~a A10/17'', respectively DAAD-project 57338573 PPP Portugal 2017, sponsored by the Federal Ministry of Education and Research (BMBF). Partial support for a visit to Exeter is gratefully acknowledged  from the Centre for Predictive Modelling in Healthcare (EPSRC grant number EP/N014391/1).

\end{document}